\documentclass[reqno, 12pt, reqno]{amsart}

\usepackage{amsfonts, amsthm, amsmath, amssymb, bbm, bm, enumerate}

\usepackage{hyperref}
\usepackage[all]{xy} 
\usepackage[margin=1.5in]{geometry}

\usepackage{helvet}

\RequirePackage{mathrsfs} \let\mathcal\mathscr

\newtheorem{thm}{Theorem}[section]
\newtheorem{theorem}[thm]{Theorem}
\newtheorem{lemma}[thm]{Lemma}

\theoremstyle{definition}

\numberwithin{equation}{section}

\renewcommand{\phi}{\varphi}
\renewcommand{\rho}{\varrho}

\newcommand{\A}{\mathbf{A}}

\newcommand{\ZZ}{\mathbb{Z}}

\newcommand{\NN}{\mathbb{N}}
\newcommand{\QQ}{\mathbb{Q}}

\renewcommand{\leq}{\leqslant}
\renewcommand{\geq}{\geqslant}

\newcommand{\m}{\mathbf{m}}

\renewcommand{\k}{\mathbf{k}}

\DeclareMathOperator{\Mod}{mod}

\newcommand{\biquad}[2]{\QQ(\sqrt{#1},\sqrt{#2})}
\newcommand{\modstar}[1]{\left(\ZZ/#1 \ZZ\right)^{\times}}

\newcommand{\legendre}[2]{\genfrac{(}{)}{}{}{#1}{#2}}

\begin{document}

\title[Hasse Norm Principle in Biquadratics]{The Hasse Norm Principle For Biquadratic Extensions}

\author{Nick Rome}
\address{School of Mathematics\\
University of Bristol \\ Bristol\\ BS8 1TW\\ UK}
\email{nick.rome@bristol.ac.uk}
\urladdr{https://people.maths.bris.ac.uk/~nr16985/}

\maketitle

\begin{abstract}
We give an asymptotic formula for the number of biquadratic extensions of the rationals of bounded discriminant that fail the Hasse norm principle.
\end{abstract}

\thispagestyle{empty}

\section{Introduction}
Let $K/k$ be a Galois extension of number fields with Galois group $G$. If $\A_K^*$ denotes the set of id{\`e}les of $K$ and $N_{K/k}$ the usual norm map then we say that $K$ satisfies the Hasse norm principle if $$N_{K/k}K^* = N_{K/k}\A_K^* \cap k^*.$$
The Hasse norm theorem states that this principle is satisfied by every cyclic extension (see e.g \cite[p.185]{CandF}). Recently, Frei--Loughran--Newton \cite{FLN} have shown  that for every non-cyclic abelian group $G$ there is an extension $K/k$ with Galois group $G$, for which the Hasse norm principle fails. Moreover, if $G \cong \mathbb{Z}/n\mathbb{Z} \oplus (\mathbb{Z}/Q\mathbb{Z})^r$ where $r \in \ZZ_{\geq 1}$ and $Q$ is the smallest prime dividing $n$ then 0\% of extensions of $k$  with Galois group $G$ fail the Hasse norm principle. However this density result is not explicit, so there is no estimate for the frequency of extensions with a given Galois group that fail the Hasse norm principle.

In this paper we'll investigate Hasse norm principle failures in the simplest abelian, non-cyclic case where $G= ( \mathbb{Z}/2\mathbb{Z})^2$ and $k = \QQ$. Let $\Delta_K$ denote the discriminant of the field $K=\biquad{a}{b}$. Note that for all such $K$, we have $\Delta_K>0$ (c.f. \eqref{disc}). We will first give a proof of the number of such fields with discriminant bounded above by $X$, recovering work of Baily \cite{Baily}. 
\begin{theorem}\label{Biquad}
Let $S(X)$ denote the number of distinct biquadratic extensions $K/\mathbb{Q}$  such that $\Delta_K \leq X$.
Then $$S(X) = \frac{23}{960} \sqrt{X} \log^2 X \prod_p \left( 1- \frac{1}{p}\right)^3\left(1+\frac{3}{p}\right)+O\left(\sqrt{X} \log X\right).$$
\end{theorem}
There have been several investigations centered on counting the number of extensions of a number field of bounded discriminant with a given Galois group, for various different choices of Galois group (see \cite{CDO} for a survey). In particular, Baily \cite[Theorem 8]{Baily} produced Theorem \ref{Biquad} in 1980, using a simple argument about how many primes could divide the discriminant. Baily's original result did not include information about the error terms although lower order terms are now known \cite[Section 2.5]{CDO}. We use a different approach to Baily and have chosen to include the proof of Theorem \ref{Biquad} as the method of computation serves to illustrate the proof of the following theorem which is the main result of this paper.
\begin{theorem}\label{BiquadHNP}
Let $\widetilde{S}(X)$ denote the number of distinct biquadratic extensions $K/\mathbb{Q}$  such that $\Delta_K \leq X$ and $K$ fails the Hasse norm principle.
Then $$\widetilde{S}(X) = \frac{1}{3\sqrt{2\pi}}\sqrt{X\log X}\, \prod_p\left( 1-\frac{1}{p}\right)^{\!\frac{3}{2}} \!\!\left(1+ \frac{3}{2p} \right)+ O\!\left(\sqrt{X}\right)\!.$$
\end{theorem}
In particular, these two results combined recover the Frei--Loughran--Newton result in this setting (i.e. that $0\%$ of biquadratic extensions of $\QQ$ fail the Hasse norm principle). 

\subsection{Layout of the paper}
In Section \ref{Criteria}, we will develop specific conditions on the integers $a$ and $b$ that ensure that the extension $K= \biquad{a}{b}$ fails the Hasse norm principle. Section 3 is devoted to the proof of the Theorem \ref{Biquad}. Theorem \ref{BiquadHNP} is proven in Section 4 using a similar approach. The main difference is that to apply the criteria developed in Section \ref{Criteria} we must sum a product of Jacobi symbols and to do this we incorporate ideas of Friedlander--Iwaniec \cite{FrandIw}.

\subsection*{Acknowledgements}
I would like to thank Daniel Loughran for suggesting this problem and my supervisor Tim Browning for bringing it to my attention, as well as for his continued advice on ways to proceed. I would also like to thank Rachel Newton for several useful conversations and R{\'e}gis de la Bret{\`e}che for numerous helpful suggestions.

\section{Criteria for Hasse Norm Principle failure}\label{Criteria}
In this section, we'll describe criteria on the integers $a$ and $b$ that determine when the extension $\QQ(\sqrt{a}, \sqrt{b})$ fails the Hasse norm principle. 
We can then sum over the $a$ and $b$ satisfying these criteria to get Theorem \ref{BiquadHNP}.

First, we describe how to count $a$ and $b$ that define unique biquadratic extensions of $\QQ$. Note that $K$ has 3 quadratic subfields
$$k_1 = \QQ(\sqrt{a}), \, k_2 = \QQ(\sqrt{b}) \text{ and } k_3 = \QQ(\sqrt{ab}/(a,b)).$$
Each of these quadratic fields can be uniquely identified by a single squarefree integer so fix $k_1 = \QQ(\sqrt{a})$ and $k_2 = \QQ(\sqrt{b})$. Let $m=(a,b)$ so that
 \begin{equation}\label{unique}a = ma_1, b= mb_1 \text{ and } (a_1,m) = (a_1, b_1) = (b_1, m) = 1.\end{equation}
Then $$k_1 = \QQ(\sqrt{ma_1}), \, k_2 = \QQ(\sqrt{mb_1}) \text{ and }k_3 = \QQ(\sqrt{a_1b_1}).$$ It is certainly true that specifying $m, a_1, b_1$ will determine $K$. Moreover since $K$ is determined uniquely by its subfields and these subfields by the choice of $m, a_1$ and $b_1$, this choice uniquely determines $K$ up to relabelling. 

We can write the discriminant of $K$ in terms of $(m,a_1,b_1)$ as follows.
By \cite[Ch. 8, 7.23]{FandT} we can express the discriminant of $K$, denoted $\Delta_K$, in terms of the discriminants of its quadratic subfields by $$\Delta_K =  \Delta_{k_1}\Delta_{k_2}\Delta_{k_3}.$$
Recall that \begin{equation}\label{quaddisc}\mathrm{disc}\left(\QQ(\sqrt{d}\,)\right) = \left\{
	\begin{array}{ll}
		d  & \mbox{if } d \equiv 1 \Mod 4 \\
		4d & \mbox{if } d \equiv 2 \text{ or } 3 \Mod 4.
	\end{array}
\right.\end{equation}
 We observe that it is not possible for just one of the integers $ma_1, mb_1$ and $a_1b_1$ to be congruent to 2 or 3 $\Mod 4$. For if $ma_1 \equiv 2 \Mod 4$ then either $m$ or $b_1 \equiv 2 \Mod 4$ and hence so is their product with $a_1$. Moreover $ma_1 \equiv 3 \Mod 4 \iff m \equiv -a_1 \Mod 4$ so either $m \equiv - b_1 \Mod 4$ or $a_1 \equiv - b_1 \Mod 4$.  Therefore \begin{equation}\label{disc}\Delta_K = c^2 m^2 a_1^2 b_1^2\end{equation} where $c$ is either 1 if all the $k_i$ are in the first case of \eqref{quaddisc}, 4 if exactly one $k_i$ is in the first case of \eqref{quaddisc} or $8$ if all the $k_i$ are in the second case.

We now turn our attention to how to identify Hasse norm principle failures and see that the congruence class of $(m, a_1, b_1) \Mod 4$ again plays a role.

\begin{lemma}\label{HNPcriteria}
Let $(m, a_1, b_1) \equiv (\epsilon_1, \epsilon_2, \epsilon_3) \Mod 4$. Then
\begin{enumerate}
\item When $\epsilon_1 = \epsilon_2 = \epsilon_3$, $K$ fails the Hasse norm principle if and only if all of the following hold:
\begin{enumerate}[(i)]
\item $ p \mid a_1 \implies \legendre{mb_1}{p} = +1$, \\
\item $ p \mid b_1 \implies \legendre{ma_1}{p} = +1$, \\
\item$ p \mid m \implies \legendre{a_1b_1}{p} = +1$.
\end{enumerate}
\item When $\epsilon_1 = \epsilon_2 \neq \epsilon_3$, $K$ fails the Hasse norm principle if and only if all of the following hold:
\begin{enumerate}[(i)]
\item $ p \mid a_1 \implies \legendre{mb_1}{p} = +1$, \\
\item $ p \mid b_1 \implies \legendre{ma_1}{p} = +1$, \\
\item$ p \mid m \implies \legendre{a_1b_1}{p} = +1$, \\
\item  $m \equiv a_1 \Mod 8$.
\end{enumerate}  Similarly for $\epsilon_2 = \epsilon_3 \neq \epsilon_1$ and $\epsilon_3 = \epsilon_1 \neq \epsilon_2$.
\item If $\epsilon_1, \epsilon_2$ and $\epsilon_3$ are pairwise distinct, then $K$ satisfies the Hasse norm principle.
\end{enumerate}
\end{lemma}
\begin{proof}
The Hasse norm principle fails in biquadratics if and only if all decomposition groups are cyclic (see e.g. \cite[Ch. 7, \S 11.4]{CandF}). Hence to come up with a criterion for Hasse norm principle failure we need to ensure that all decomposition groups are proper subgroups of the Galois group $(\mathbb{Z}/2\mathbb{Z})^2$. Therefore we need every rational prime to split in $K$.

A prime splits in $K$ if and only if it splits in at least one of the three quadratic subfields $\mathbb{Q}(\sqrt{ma_1})$, $\mathbb{Q}(\sqrt{mb_1})$ and $\mathbb{Q}(\sqrt{a_1b_1})$. 

\begin{enumerate}
\item If $a_1 \equiv m \equiv b_1$ mod 4 then the only primes that ramify in $\mathbb{Q}(\sqrt{ma_1})$ are those dividing $ma_1$, therefore since $m, a_1$ and $b_1$ are pairwise coprime, no prime ramifies in all three quadratic subextensions. A prime $p$ splits in $\mathbb{Q}(\sqrt{a}) \iff \legendre{a}{p} = +1$ otherwise it remains inert.
This means a prime cannot be inert in all three subfields. Hence we must ensure that all primes that ramify in two of the subfields split in the third.
\item If $m \equiv a_1 \not \equiv b_1 \Mod 4$ then $ma_1 \equiv 1 \Mod 4$, $mb_1 \equiv 2 \text{ or } 3 \Mod 4$ and $a_1b_1 \equiv 2 \text{ or } 3 \Mod 4$. Note that since $m$ and $a_1$ are coprime they cannot both be congruent to $2$, so they must be odd. We see that the rational prime 2 also ramifies in the subfields $\mathbb{Q}(\sqrt{mb_1})$ and  $\mathbb{Q}(\sqrt{a_1b_1})$. Therefore we must also ensure that 2 splits in $\mathbb{Q}(\sqrt{ma_1})$ so we must impose the extra condition $ma_1 \equiv 1$ mod 8.
\item In this case we have $ma_1 \equiv 2 \text { or }3 \Mod 4$, $mb_1 \equiv 2 \text{ or } 3 \Mod 4$ and $a_1b_1 \equiv 2 \text{ or } 3 \Mod 4$  so 2 ramifies in all 3 quadratic subextensions hence is totally ramified in $K$. Therefore the Hasse norm principle holds.
\end{enumerate}
\end{proof}

\section{Proof of Theorem \ref{Biquad}}\label{Counting1}
We move on to establishing the asymptotic formula for the number of biquadratic extensions $K=\mathbb{Q}(\sqrt{a}, \sqrt{b})$ of bounded discriminant. In Section \ref{Criteria} we saw that this means counting the number of integer triples which are square-free, pairwise coprime and whose product is bounded above in a way that depends on the congruence class of the tuple $\Mod 4$. In Section \ref{Criteria}, these triples were denoted $(m,a_1,b_1)$ however from here on out we will write them $(m_1,m_2,m_3)$ for simplicity.

 Let  $\bm{\delta}=(\delta_2, \delta_3)$ where $\delta_2$ denotes the sign of $m_2$ and $\delta_3$ the sign of $m_3$. We note that it would be redundant to keep track of all 3 signs and just 2 will suffice.
Observe that the highest power of 2 dividing $m_1m_2m_3$ is either 0 or 1.
To keep track of this we write $$m_1 = 2^{\mu}m_1', \, m_2 = \delta_2 2^{\alpha}m_2' \, \text{ and } m_3 = \delta_32^{\beta}m_3',$$ where $2 \nmid m_1'm_2'm_3'$. Finally we denote by $\bm{\epsilon}$ the congruence class of $(m_1', m_2', m_3')$ $\Mod 4$.
Then 
$$S(X) = \frac{1}{6}\sum_{\bm{\delta} \in \{\pm1\}^2}\sum \limits_{\substack{\mu + \alpha + \beta \in \{0,1\}\\ \mu,\alpha,\beta \in \{0,1\}}} \sum_{\bm{\epsilon} \in \{\pm1\}^3} T(\bm{\delta}, \bm{\epsilon}, \mu, \alpha, \beta),$$
where $T(\bm{\delta}, \bm{\epsilon}, \mu, \alpha, \beta)$ counts the number of tuples $(m_1', m_2', m_3') \in \NN^3$ such that the following all hold:
\begin{enumerate}[i)]
\item$\mu^2(m_1'm_2'm_3') =1$,\\
\item $(m_1',m_2',m_3') \equiv \bm{\epsilon} \Mod 4$, \\
\item $2^{\mu + \alpha + \beta}c_{\bm{\delta},\bm{\epsilon}, \mu, \alpha, \beta}m_1'm_2'm_3'\leq \sqrt{X}$.
\end{enumerate}
Note that $T(\bm{\delta},\bm{\epsilon}, \mu, \alpha, \beta)$ overcounts each triple $(m_1', m_2', m_3')$ by counting every permutation of the components so we divide the sum by 6 to compensate.
The constants $c_{\bm{\delta},\bm{\epsilon}, \mu, \alpha, \beta}$ correspond to the constant $c$ in \eqref{disc}. Recall that $c=1$ if all the components of $(m_1, m_2, m_3)$ are congruent $\Mod 4$, $c=4$ if two of the components are congruent, and 8 otherwise. We summarise the values below
$$\begin{array}{ll} c_{\bm{\delta},\bm{\epsilon},0,0,0} = \left\{
	\begin{array}{ll}
		1 & \mbox{if }\epsilon_1 =\delta_2\epsilon_2 =\delta_3\epsilon_3 ,\\
		4 & \mbox{otherwise};
	\end{array}
\right.  & c_{\bm{\delta},\bm{\epsilon},1,0,0} = \left\{
	\begin{array}{ll}
		4 & \mbox{if }\delta_2\epsilon_2 =\delta_3\epsilon_3 ,\\
		8 & \mbox{otherwise};
	\end{array}
\right. \\  c_{\bm{\delta},\bm{\epsilon},0,1,0} = \left\{
	\begin{array}{ll}
		4 & \mbox{if }\epsilon_1 =\delta_3\epsilon_3, \\
		8 & \mbox{otherwise};
	\end{array}
\right.  & c_{\bm{\delta},\bm{\epsilon},0,0,1} = \left\{
	\begin{array}{ll}
		4 & \mbox{if }\epsilon_1 =\delta_2\epsilon_2, \\
		8 & \mbox{otherwise}.
	\end{array}
\right. 
\end{array}$$
We'll first tackle the evaluation of $T(\bm{\delta},\bm{\epsilon}, \mu, \alpha, \beta)$, for fixed $\bm{\delta},\bm{\epsilon}, \alpha, \beta$ and $\mu$ then the remaining summation will be a simple computation.
\begin{lemma}\label{main1}
Let $c = c_{\bm{\delta},\bm{\epsilon},\mu, \alpha, \beta}$. Then we have
\begin{align*}T(\bm{\delta}, \bm{\epsilon}, \mu, \alpha, \beta)=  \frac{\sqrt{X}\log^2X}{160c2^{\mu + \alpha + \beta}} \prod_p \left( 1- \frac{1}{p}\right)^3\left(1+\frac{3}{p}\right)+O\left(\sqrt{X} \log X\right).\end{align*}
\end{lemma}
\begin{proof}
To ease notation, we will denote $N= \sqrt{X}/\left(c2^{\mu + \alpha + \beta}\right)$.
Then,
$$ T(\bm{\delta}, \bm{\epsilon}, \mu, \alpha, \beta)= \sum \limits_{\substack{m_1'm_2'm_3' \leq N\\ m_i' \equiv \epsilon_i \Mod 4 }} \mu^2(m_1'm_2'm_3').$$
We first remove the congruence condition by applying the character sum $$\bm{1}_{\{m_i' \equiv \epsilon_i \Mod 4\}} = \frac{1}{2} \sum_{\nu_i \in \{1,2\}} \chi^{\nu_i}(m_i'\epsilon_i),$$ where $\chi$ is the non-principal character mod 4. Then,
\begin{align*}
T(\bm{\delta}, \bm{\epsilon}, \mu, \alpha, \beta) &= \frac{1}{8} \sum_{m_1'm_2'm_3' \leq N} \mu^2(m_1'm_2'm_3') \sum_{\bm{\nu} \in \{1,2\}^3} \prod_{i \in \{1,2,3\}} \chi^{\nu_i}(m_i'\epsilon_i)\\
&= \frac{1}{8}  \sum_{\bm{\nu} \in \{1,2\}^3} \prod_{i} \chi^{\nu_i}(\epsilon_i) \!\!\! \sum \limits_{\substack{m_1'm_2'm_3'\leq N }}\!\! \mu^2(m_1'm_2'm_3')   \prod_{i} \chi^{\nu_i}(m_i').
\end{align*}

We investigate $$\sum_{n \leq N} a_n := \sum_{n \leq N}\mu^2(n) \sum \limits_{m_1m_2m_3 =n} \chi^{\nu_1}(m_1)\chi^{\nu_2}(m_2)\chi^{\nu_3}(m_3) $$ by looking at the associated Dirichlet series $$F(s) := \sum_n \frac{a_n}{n^s} .$$
Comparing this to the product of the $L$-functions associated to $\chi^{\nu_i}$ we introduce $$G(s) := F(s) \prod_{i=1}^3L(\chi^{\nu_i},s)^{-1},$$ where $L(\chi^2,s) = \left(1-\frac{1}{2^s}\right) \zeta(s).$ For example, when $\nu_i = 2$ for all $i$ then $$G(s) = \left( 1-\frac{1}{2^s}\right)^3\left(1+\frac{3}{2^s}\right)^{-1} \prod_p \left(1 + \frac{3}{p^s}\right)\left(1-\frac{1}{p^s}\right)^3.$$
For any $\nu_i$, the Euler product $G(s)$ is absolutely convergent for $\text{Re}(s)>\frac{1}{2}$ and in this region $G(s) \ll 1.$ The $L$-function associated to $\chi$ is entire and in the region $\frac{1}{2} < \sigma \leq 1$, we have $$L(\chi, \sigma+it) \ll_{\epsilon} \vert t \vert^{\frac{1-\sigma}{2}},$$ see for example \cite[5.20]{IwK}. Therefore the Dirichlet series $F(s)$ satisfies the conditions of the Selberg--Delange theorem \cite[II.5.2, Theorem 5.2]{Ten}. When $\nu_i=2$ for all $i$, the $L$-functions each correspond to a copy of $\zeta(s)$ and therefore the sum has order of magnitude $N \log^2 N$. In each other case, there are at most 2 copies of $\zeta(s)$ and hence the contribution is $O\left( N \log N\right).$
We conclude that $$\sum_{n \leq N} a_n = c N \log^2N + O(N\log N),$$ where $$c= \frac{(1-\frac{1}{2})^{-3}}{\Gamma(3)}G(1) = \frac{1}{5} \prod_p \left(1+\frac{3}{p}\right)\left(1-\frac{1}{p}\right)^3.$$

\end{proof}
All that's left is to apply the definition of $c_{\bm{\delta}, \bm{\epsilon},\mu,\alpha,\beta}$ in each case and find that
$$\sum_{\bm{\delta}\in \{\pm1\}^2}\sum_{\mu + \alpha + \beta \in \{0,1\}} \sum_{\bm{\epsilon} \in \{\pm 1\}^3} \frac{1}{c_{\bm{\delta}, \bm{\epsilon},\mu,\alpha,\beta}2^{\mu+\alpha+\beta}} = 23.$$

\section{Proof of Theorem \ref{BiquadHNP}}

Similarly to the previous section, we start by making the change of variables \begin{equation} \label{variables} m_1 = 2^{\mu}m_1', \, m_2 = \delta_2 2^{\alpha}m_2' \, \text{ and } m_3 = \delta_32^{\beta}m_3',\end{equation} where $\mu^2(2m_1'm_2'm_3') =1$ and $\mu,\beta,\alpha \in \ZZ_{\geq 0}$ such that $\mu+\alpha + \beta \leq 1.$ We saw in Section \ref{Criteria} that when counting Hasse norm principle failure it is important to keep track of the residue class of $(m_1', m_2', m_3') \Mod 8$ rather than just mod 4 as in Section \ref{Counting1}. Recall from Section \ref{Criteria} that if the congruence classes of $m_1, m_2$ and $m_3 \Mod 4$ are all distinct then $K$ always satisfies the Hasse norm principle. Moreover if exactly two of them are congruent mod 4 then we require that these two are in fact congruent mod 8 to ensure Hasse norm principle failure. Hence we restrict our attention to such classes which we denote $E(\mu, \alpha, \beta)$. Specifically
\begin{align*}
 E(\bm{0}) & = E_1(\bm{0}) \cup E_2(\bm{0}),
\end{align*}
where 
\begin{align*}
E_1(\bm{0}) &:= \{\bm{\epsilon} \in ((\ZZ/8\ZZ)^{\times})^3: \epsilon_1 \equiv \epsilon_2 \equiv  \epsilon_3 \Mod 4\}; \\
E_2(\bm{0}) &:= \bigcup \limits_{\substack{i,j,k \\ \text{pairwise distinct}}} \{ \bm{\epsilon} \in (\modstar{8})^3 : \epsilon_i = \epsilon_j \equiv - \epsilon_k \Mod 4\},
\end{align*} and 
\begin{align*}
E(1,0,0) &= \{\bm{\epsilon} \in ((\ZZ/8\ZZ)^{\times})^3 : \epsilon_2 = \epsilon_3\};\\
E(0,1,0) &= \{\bm{\epsilon} \in ((\ZZ/8\ZZ)^{\times})^3 : \epsilon_1 = \epsilon_3\};\\
E(0,0,1) &= \{\bm{\epsilon} \in ((\ZZ/8\ZZ)^{\times})^3 : \epsilon_1 = \epsilon_2\}.
\end{align*}

 Analogously to Section \ref{Counting1}, we define the constants $c_{\bm{\delta}, \bm{\epsilon}, \mu, \alpha, \beta}$ to account for the different discriminants in each case by setting $$c_{\bm{\delta},\bm{\epsilon}, \mu, \alpha, \beta} = \left\{
	\begin{array}{ll}
		1 &\text{ if }(\epsilon_1, \delta_2 \epsilon_2, \delta_3 \epsilon_3) \in E_1(\bm{0}) \text{ and } \mu =0 =\alpha = \beta\\
		4 & \text{ otherwise.}
	\end{array}
\right. $$Then 
\begin{equation}\label{simpleHNP}\widetilde{S}(X) = \frac{1}{6} \sum_{(\delta_2, \delta_3) \in \{\pm1\}^2} \sum \limits_{\substack{\mu+ \alpha + \beta \in \{0,1\}\\ \mu,\alpha,\beta \in \{0,1\}}} \sum_{(\epsilon_1, \delta_2 \epsilon_2,\delta_3 \epsilon_3) \in E(\mu, \alpha, \beta)}\widetilde{T}(\bm{\delta},\bm{\epsilon}, \mu, \alpha, \beta),\end{equation}
where $\widetilde{T}(\bm{\delta}, \bm{\epsilon}, \mu, \alpha, \beta)$ counts the number of tuples $(m_1', m_2', m_3') \in \mathbb{N}^3$ such that the following all hold:
\begin{enumerate}[i)]
\item $\mu^2(m_1'm_2'm_3') =1$,
\item $(m_1', m_2', m_3')\equiv \bm{\epsilon} \Mod 8$,
\item $2^{\mu+ \alpha + \beta}c_{\bm{\delta}, \bm{\epsilon}, \mu, \alpha, \beta}m_1'm_2'm_3' \leq \sqrt{X}$,
\item $(m_1', m_2', m_3')$ satisfies the local conditions for Hasse norm principle failure in Lemma \ref{HNPcriteria}.
\end{enumerate}

Observe that condition (iv) can be detected by the following indicator function
$$ \prod_{p\mid m_1'} \frac{1}{2} \left(1+\! \legendre{m_2m_3}{p}\!\right)\prod_{p\mid m_2'} \frac{1}{2} \left(1+\! \legendre{m_1m_3}{p}\!\right)\prod_{p\mid m_3'} \frac{1}{2} \left(1+\! \legendre{m_1m_2}{p}\!\right).$$
Writing $M= \frac{\sqrt{X}}{ 2^{\mu+\alpha+\beta}c_{\bm{\delta}, \bm{\epsilon}, \mu, \alpha, \beta}}$ this means that we may express $\widetilde{T}(\bm{\delta}, \bm{\epsilon}, \mu, \alpha, \beta)$ as
$$\sum \limits_{\substack{m_1'm_2'm_3'\leq M\\m_i' \equiv \epsilon_i \Mod 8}} \frac{\mu^2(m_1'm_2'm_3')}{\tau(m_1'm_2'm_3')} \!\!\prod_{p \mid m_1'm_2'm_3'}\!\!\! \left(\!\!1+ \!\legendre{m_1m_2}{p}\!\!\right)\!\!\left(\!\!1+ \!\legendre{m_2m_3}{p}\!\!\right)\!\!\left(\!\!1+\! \legendre{m_3m_1}{p}\!\!\right)\!\! .$$
By expanding out the product, we clearly have
$$\widetilde{T}(\bm{\delta}, \bm{\epsilon}, \mu, \alpha, \beta)  = \sum \limits_{\substack{m_1'm_2'm_3'\leq M\\m_i' \equiv \epsilon_i \Mod 8\\ m_i' = k_i \tilde{k}_i}}\!\! \frac{\mu^2(m_1'm_2'm_3')}{\tau(m_1'm_2'm_3')}  \legendre{m_1m_2}{k_3}\legendre{m_2m_3}{k_1}\legendre{m_3m_1}{k_2}.$$
Recalling \eqref{variables}, we see that $$\legendre{\! m_1m_2\!}{k_3}\!\legendre{\! m_2m_3\!}{k_1}\!\legendre{\! m_3m_1\!}{k_2}\! =\! \legendre{2^{\mu}}{k_2k_3}\!\legendre{2^{\alpha}\delta_2}{k_3k_1}\!\legendre{2^{\beta}\delta_3}{k_1k_2}\!\legendre{m_1'm_2'}{k_3}\!\legendre{m_2'm_3'}{k_1}\!\legendre{m_3'm_1'}{k_2}\!.$$
We may now repeatedly apply the law of quadratic reciprocity to conclude
\begin{equation}\label{finalT}
\widetilde{T}(\bm{\delta}, \bm{\epsilon}, \mu, \alpha, \beta) =\!\!\!\! \sum \limits_{\substack{m_1'm_2'm_3' \leq M\\m_i' \equiv \epsilon_i \Mod 8\\ m_i' = k_i \tilde{k}_i }} u(\k) \frac{\mu^2(m_1'm_2'm_3')}{\tau(m_1'm_2'm_3')} \legendre{\tilde{k}_1}{k_2k_3} \legendre{\tilde{k}_2}{k_3k_1} \legendre{\tilde{k}_3}{k_1k_2} ,
\end{equation}
where $ u(\k) = (-1)^{\nu(k_1)\nu(k_2) + \nu(k_2)\nu(k_3) + \nu(k_3) \nu(k_1)}\legendre{2^{\mu}}{k_2k_3}\legendre{2^{\alpha}\delta_2}{k_3k_1}\legendre{2^{\beta}\delta_3}{k_1k_2}.$ Here
  $\nu(b)$ is defined to be 0 if $b \equiv 1 \Mod 4$ and 1 otherwise, for any odd integer $b$.

These character sums are strongly reminiscent of the ones studied by Friedlander and Iwaniec in \cite{FrandIw}, and we will follow their approach to evaluate them, making use of the following results.
\begin{lemma}\label{FRIW1}
If $q=q_1q_2$ where $(q_1, q_2) =1$, $(ad,q)=1$ and $\chi_2$ is a non-principal character modulo $q_2$. Then for any $C>0$ we have
\begin{align*}\sum \limits_{\substack{n \leq x \\ (n,d) =1 \\ n \equiv a \Mod q_1}} \mu^2(n) \frac{\chi_2(n)}{\tau(n)} &\ll_C \tau(d) qx (\log x)^{-C}.\end{align*}
\end{lemma}
\begin{proof}
This is the error term in \cite[Corollary 2]{FrandIw}.
\end{proof}

\begin{lemma}\label{FRIW2'}
Let $\alpha_m$, $\beta_n$ be any complex numbers supported on odd integers with modulus bounded by one. Then
$$\mathop{\sum \sum} \limits_{\substack{m,n  > V\\ mn \leq X}} \alpha_m \beta_n \legendre{m}{n} \ll XV^{-\frac{1}{6}}  (\log X)^{\frac{7}{6}}.$$
\end{lemma}

\begin{proof}
We break the range of summation into dyadic intervals and then apply \cite[Lemma 2]{FrandIw} in each interval. Thus
\begin{align*}
\left| \mathop{\sum\sum}\limits_{\substack{m,n  > V\\ mn \leq X}} \alpha_m \beta_n \legendre{m}{n} \right| &\leq \mathop{\sum \sum} \limits_{\substack{2^i, 2^j > V \\ 2^{i+j} \leq X}} \left| \sum_{2^i < m \leq 2^{i+1}} \sum_{2^j < n \leq 2^{j+1}}  \alpha_m \beta_n \legendre{m}{n} \right| \\
&\ll \mathop{\sum\sum}\limits_{\substack{2^i, 2^j > V \\ 2^{i+j} \leq X}}  2^{i+j}(2^{-\frac{i}{6}} + 2^{-\frac{j}{6}} )  (i+j)^{\frac{7}{6}}\\
&\ll (\log X)^{\frac{7}{6}} \mathop{\sum\sum}\limits_{\substack{2^i, 2^j > V \\ 2^{i+j} \leq X}} 2^{i + \frac{5j}{6}}\\
&=  (\log X)^{\frac{7}{6}} \sum_{V \leq 2^i \leq X/V} 2^i \sum_{V \leq 2^j \leq X/2^i}2^{\frac{5j}{6}} \ll XV^{-\frac{1}{6}} (\log X)^{\frac{7}{6}} .
\end{align*}
\end{proof}

The estimate in Lemma \ref{FRIW1} 
 is only useful when the modulus of the character in our sum is smaller than a power of $\log X$. Set $$V=(\log X)^B,$$ for a large constant parameter $B$ at our disposal. The moduli of the characters involved in \eqref{finalT} can either be thought of as being the $k_i$ or the $\tilde{k}_i$. We get a main term contribution from Lemma~\ref{FRIW1} in the ranges where each of the three characters has small modulus, hence there are three possibilities:
\begin{enumerate}[(i)]
\item$ k_i \leq V \text{ for all } i.$
\item $\tilde{k}_i \leq V \text{ for all } i.$
\item $\tilde{k}_i, \tilde{k}_j, k_i, k_j \leq V \text{ for some choice of } i \neq j.$
\end{enumerate}

In every other case, we have $k_i, \tilde{k}_j > V$ for some $i \neq j$ and we can apply Lemma~\ref{FRIW2'} to exploit cancellation in the $\legendre{\tilde{k}_i}{k_j}$ term and get a negligible contribution. To illustrate this, suppose that $k_2, \tilde{k}_1 > V$. Then we can apply Lemma \ref{FRIW2'} to the factor $\legendre{\tilde{k}_1}{k_2}$ in \eqref{finalT}. Let \begin{align*}
f_1(\tilde{k}_1) &= \mathbf{1}_{\left\{ \substack{\tilde{k}_1 \equiv \epsilon_1 k_1 \Mod 8\\ (\tilde{k}_1,2k_1\tilde{k}_2k_3\tilde{k}_3)=1}\right\}}\frac{\mu^2(\tilde{k}_1)}{\tau(\tilde{k}_1)} \legendre{\tilde{k}_1}{k_3}\\
f_2(k_2) &= \mathbf{1}_{\left\{ \substack{k_2 \equiv \epsilon_2 \tilde{k}_2 \Mod 8\\ (k_2,2k_1\tilde{k}_2k_3\tilde{k}_3)=1}\right\}}\frac{\mu^2(k_2)}{\tau(k_2)} \legendre{\tilde{k}_3}{k_2}.
\end{align*} Then the $k_2$, $\tilde{k}_1$ sum is given by
\begin{align*}
& \sum \limits_{\substack{k_2, \tilde{k}_1 \geq V \\ k_2\tilde{k}_1 \leq \frac{\sqrt{X}}{k_1\tilde{k}_3k_3\tilde{k}_2}\\(\tilde{k}_1,k_2)=1}}f_1(\tilde{k}_1)f_2(k_2)u(\k) \legendre{\tilde{k}_1}{k_2}\\
&=  \sum \limits_{\substack{d \leq \frac{X^{\frac{1}{4}}}{\sqrt{k_1\tilde{k}_3k_3\tilde{k}_2}}\\(d,2k_1\tilde{k}_2k_3\tilde{k}_3)=1}}\!\!\!\mu(d)\!\!\!\sum \limits_{\substack{k_2, \tilde{k}_1 \geq V/d \\ k_2\tilde{k}_1 \leq \frac{\sqrt{X}}{d^2k_1\tilde{k}_3k_3\tilde{k}_2}}}f_1(\tilde{k}_1)\legendre{\tilde{k}_1}{d}f_2(k_2)\legendre{d}{k_2}u(k_1,dk_2,k_3) \legendre{\tilde{k}_1}{k_2}\\
 &\ll \sqrt{X}V^{-\frac{1}{6}}  ( \log X)^{\frac{7}{6}}  \frac{\tau(\tilde{k}_2k_1k_3 \tilde{k}_3)}{\tilde{k}_2k_1k_3 \tilde{k}_3}.
\end{align*}
When summed trivially over the remaining variables, this leads to a contribution of $$\ll \sqrt{X}V^{-\frac{1}{6}}  ( \log X)^{\frac{31}{6}}.$$
Of course, this bound applies for all ranges $k_i, \tilde{k}_j >V$ where $i \neq j$.\\

From here on out then we will restrict to the three different ranges mentioned above.
First, consider the range $\tilde{k}_i \leq V$ for $i=1,2,3$ and fix $k_1, \tilde{k}_1, \tilde{k}_2$ and $\tilde{k}_3$. In order to evaluate $\widetilde{T}(\bm{\delta},\bm{\epsilon}, \mu, \alpha, \beta)$ we need to look at 
$$ T_{k_1, \tilde{k}_1, \tilde{k}_2, \tilde{k}_3} = \mathop{\sum \sum}  \limits_{\substack{k_2k_3 \leq M/m_1'\tilde{k}_2\tilde{k}_3\\ k_i \equiv \epsilon_i \tilde{k}_i \Mod 8 \text{ for } i=2,3 \\ (k_2k_3,2m_1'\tilde{k}_2\tilde{k}_3) =1}} \frac{\mu^2(k_2k_3)}{\tau(k_2k_3)} \legendre{\tilde{k}_3}{k_2}\legendre{\tilde{k}_2}{k_3}\legendre{\tilde{k}_1}{k_2k_3}u(\k),$$
then compute the sum \begin{equation}\label{firstrange} \mathop{\sum \sum \sum} \limits_{\substack{k_1\tilde{k}_1\tilde{k}_2\tilde{k}_3\leq M\\\tilde{k}_i \leq V \\ m_1' \equiv \epsilon_1 \Mod 8}}\frac{\mu^2(2m_1'\tilde{k}_2\tilde{k}_3)}{\tau(m_1'\tilde{k}_2\tilde{k}_3)} \legendre{\tilde{k}_2\tilde{k}_3}{k_1}T_{k_1, \tilde{k}_1, \tilde{k}_2, \tilde{k}_3}.\end{equation}
Note here that we can pull $u(\k)$ out of the expression for $T_{k_1, \tilde{k}_1, \tilde{k}_2, \tilde{k}_3}$.
Indeed, $\nu(x)$ only depends on the residue class of $x$ mod $4$, and $\legendre{2}{x}$ is determined by the residue class of $x$ mod $8$. Therefore
$$u(\k) = u(\epsilon_1\tilde{k}_1, \epsilon_2 \tilde{k}_2, \epsilon_3 \tilde{k}_3).$$
Analogously in the range  $k_i \leq V$ for $i=1,2,3$, we need to evaluate  \begin{equation}\label{secondrange} \mathop{\sum \sum \sum} \limits_{\substack{\tilde{k}_1k_1k_2k_3\leq M\\ k_i \leq V \\ m_1' \equiv \epsilon_1 \Mod 8}}\frac{\mu^2(2m_1'k_2k_3)}{\tau(m_1'k_2k_3)} \legendre{\tilde{k}_1}{k_2k_3}T'_{k_1, \tilde{k}_1, k_2, k_3},\end{equation}
where $$T'_{\tilde{k}_1,k_1, k_2, k_3} =\mathop{\sum \sum} \limits_{\substack{\tilde{k}_2 \tilde{k}_3 \leq M/m_1'k_2k_3\\ k_i \equiv \epsilon_i \tilde{k}_i \Mod 8 \text{ for } i=2,3 \\ (\tilde{k_2}\tilde{k}_3,2m_1'k_2k_3) =1}} \frac{\mu^2(\tilde{k}_2\tilde{k}_3)}{\tau(\tilde{k}_2\tilde{k}_3)} \legendre{\tilde{k}_3}{k_2}\legendre{\tilde{k}_2}{k_3}\legendre{\tilde{k}_2\tilde{k}_3}{k_1}u(\k) .$$

Our approach will be to use Lemma \ref{FRIW1} to handle the contribution to the sums $T_{k_1, \tilde{k}_1, \tilde{k}_2, \tilde{k}_3}$ and $T'_{\tilde{k}_1,k_1, k_2, k_3}$ arising from non-principal characters, then compute the main terms where each character is principal. In the latter range, the reciprocity factor in the main term is $u(1,1,1)=1$. However in the former, we have the term $u(\bm{\epsilon})$ which when summed over all possible $\bm{\epsilon}$ will lead to cancellation (c.f. Lemma \ref{ucsum}).

Finally, we consider the third type of range. Suppose $k_2, k_3, \tilde{k}_2, \tilde{k}_3 \leq V$ then we will fix all four of these variables and evaluate
$$T''_{k_2, \tilde{k}_2, k_3, \tilde{k}_3} = \mathop{\sum \sum} \limits_{\substack{k_1\tilde{k}_1 \leq M/m_2'm_3'\\ k_1 \equiv \epsilon_1 \tilde{k}_1 \Mod 8 \\ (k_1\tilde{k}_1,2m_2'm_3')=1}}u(\k) \frac{\mu^2(k_1\tilde{k}_1)}{\tau(k_1\tilde{k}_1)} \legendre{\tilde{k}_2\tilde{k}_3}{k_1} \legendre{\tilde{k}_1}{k_2k_3},$$ then compute
$$\sum \limits_{\substack{k_2\tilde{k}_2k_3\tilde{k}_3 \leq \min \{M, V\}\\ k_i \equiv \epsilon_i \tilde{k}_i \text{ for } i = 2,3}} \frac{\mu^2(k_2\tilde{k}_2k_3\tilde{k}_3)}{\tau(k_2\tilde{k}_2k_3\tilde{k}_3)}\legendre{\tilde{k}_2}{k_3}\legendre{\tilde{k}_3}{k_2}T''_{k_2, \tilde{k}_2, k_3, \tilde{k}_3}.$$

Before attempting to evaluate $T''_{k_2, \tilde{k}_2, k_3, \tilde{k}_3}$, we wish to remove the $u(\k)$ term as in the other ranges. However this is marginally more complicated than in the other cases. For fixed $k_2$ and $k_3$, we may think of $u(\k)$ as a character on $k_1$. Our main term contribution will occur when this character is principal, in particular this requires $k_2=k_3=1$. In that case $$u(k_1, 1, 1) = \legendre{\delta_2\delta_3 2^{\alpha + \beta}}{k_1}.$$ Hence one will only get a main term contribution when $k_2 = k_3 =\tilde{k}_2 = \tilde{k}_3 =1$ and $\delta_2\delta_3 2^{\alpha + \beta} =1$. This gives a further constraint on the outer sum over $\bm{\delta}, \mu, \alpha$ and $\beta$. The following lemma deals with the first two types of ranges.

\begin{lemma}\label{Tsum}
If $\tilde{k}_i \leq V$ for $i=1,2,3$ then for any $C>0$, we have 
\begin{align*}
T_{k_1, \tilde{k}_1, \tilde{k}_2, \tilde{k}_3} = \delta_{\tilde{k}_1=\tilde{k}_2=\tilde{k}_3=1}\frac{  u(\bm{\epsilon} )M}{4k_1\pi^2}& \prod_{p \mid k_1}\!\left(1+\frac{1}{p}\right)^{-1}\!\left\{ 1+ O\left(\tau(k_1)\sqrt{\frac{k_1}{M}}\right)\right\}\!\\
&+O_C(MV^2\log V(\log M)^{2-C}).\end{align*} Analogously, if $k_i \leq V$ for $i=1,2,3$ then for any $C>0$, 
\begin{align*}
T'_{\tilde{k}_1,k_1, k_2, k_3} = \delta_{k_1=k_2=k_3=1}\frac{  u(1,1,1 )M}{4\tilde{k}_1\pi^2} &\prod_{p \mid \tilde{k}_1}\!\left(1+\frac{1}{p}\right)^{-1}\!\left\{ 1+ O\left(\tau(\tilde{k}_1)\sqrt{\frac{\tilde{k}_1}{M}}\right)\!\!\right\}\!\\
&+O_C(\sqrt{M}V^2\log V(\log M)^{-C}).\end{align*}

\end{lemma}
Here by $\delta_{\tilde{k}_1=\tilde{k}_2=\tilde{k}_3=1}$ we simply mean the function that is 1 when all $\tilde{k}_i = 0$ and 0 otherwise. Hence, as expected, our main term corresponds to the case when all characters are principal. The final set of ranges is dealt with by the next lemma.

\begin{lemma}\label{weird}
If $k_2, \tilde{k}_2, k_3, \tilde{k}_3 \leq V$ then for any $C>0$, we have \begin{align*}T''_{k_2, \tilde{k}_2, k_3, \tilde{k}_3} = \delta_{k_2 = \tilde{k}_2 = k_3 = \tilde{k}_3= \delta_2\delta_32^{\alpha+\beta} = 1} \frac{M}{\pi^2}&\left\{ 1+ O\left(\frac{1}{\sqrt{M}}\right)\!\!\right\} \\ &+O_C(\sqrt{M}V^2(\log V)^2(\log M)^{-C}).\end{align*} The ranges $k_1, \tilde{k}_1, k_2, \tilde{k}_2 \leq V$ and $k_1, \tilde{k}_1, k_3, \tilde{k}_3 \leq V$ are analagous.
\end{lemma}
The proof of this follows the proof of Lemma \ref{Tsum}. Observe that these ranges differ from those in Lemma \ref{Tsum} since in those there is a final variable ($k_1$ and $\tilde{k}_1$ respectively) that needn't be made 1 in our main term. In performing the sum over this last variable we pick up a factor of $\sqrt{\log X}$ which is absent in the ranges considered in Lemma \ref{weird} hence these ranges give a smaller contribution.

\begin{proof}[Proof of Lemma \ref{Tsum}]
We will focus on the case when $\tilde{k}_i \leq V$ . The other case follows the same argument.

The first step will be to remove the congruence conditions on $k_2$ and $k_3$ by using a sum of characters mod 8. Performing this and re-arranging, we may write $T_{k_1,\tilde{k}_1,\tilde{k}_2,\tilde{k}_3} $ as
$$\frac{u(\bm{\epsilon}\tilde{\k})}{16}\sum_{\chi_i \Mod 8} \chi_1(\epsilon_2) \chi_2(\epsilon_3)\!\!\!\!\!\!\!\sum \limits_{\substack{k_2k_3 \leq M/k_1\tilde{k}_1\tilde{k}_2\tilde{k}_3 \\ (k_2k_3, 2m_1'\tilde{k}_2\tilde{k}_3)=1}} \frac{\mu^2(k_2k_3)}{\tau(k_2k_3)} \legendre{\tilde{k}_2}{k_3}\legendre{\tilde{k}_3}{k_2}\legendre{\tilde{k}_1}{k_2k_3}\chi_1(k_2)\chi_2(k_3),$$ where $u(\bm{\epsilon}\tilde{\k}) = u(\epsilon_1\tilde{k}_1, \epsilon_2\tilde{k}_2, \epsilon_3\tilde{k}_3)$. We gather together the two characters on $k_2$,$k_3$ into new characters labelled $\widetilde{\chi}_i$. 
If either of these characters is non-principal then we can use Lemma \ref{FRIW1} to get a bound. Suppose that $\widetilde{\chi}_2$ is non-principal. Then for any $C>0$, the total sum is bounded by
\begin{align*}
&\ll \tau(m_1'\tilde{k}_2\tilde{k}_3)\frac{M}{k_1\tilde{k}_3}(\log M/\tilde{k}_1\tilde{k}_2\tilde{k}_3)^{-C}\sum_{k_2 \leq M/\tilde{k}_1\tilde{k}_2\tilde{k}_3} \frac{\mu^2(k_2)}{k_2}\vert\widetilde{\chi}_1(k_2)\vert\\
 &\ll \tau(m_1'\tilde{k}_2\tilde{k}_3)\frac{M}{k_1\tilde{k}_3}(\log M/\tilde{k}_1\tilde{k}_2\tilde{k}_3)^{1-C}.
\end{align*}
When summed trivially over the remaining variables this gives an error term of size $$\ll \sqrt{X}V^2\log V(\log X)^{2-C}.$$
This bound is definitely not best possible, but all that matters is we have an arbitrary $\log$ power saving over our main term. 

We turn now to estimating the main term which is given by \begin{align*}
\frac{u(\bm{\epsilon})}{16}\sum \limits_{\substack{k_2k_3 \leq M/k_1\\ (k_2k_3, 2k_1)=1}} \frac{\mu^2(k_2k_3)}{\tau(k_2k_3)} &= \frac{u(\bm{\epsilon})}{16}\sum \limits_{\substack{n \leq M/k_1 \\ (n,2k_1)=1}} \frac{\mu^2(n)}{\tau(n)} \sum_{n= k_2k_3} 1\\
&=\frac{u(\bm{\epsilon})}{16} \sum \limits_{\substack{n \leq M/k_1 \\ (n,2k_1)=1}}  \mu^2(n).
\end{align*}
The lemma now immediately follows from the simple identity
$$\sum_{\substack{x \leq X \\ (x,d)=1}} \mu^2(x) = \frac{6}{\pi^2} X \prod_{p \mid d} \left(1+\frac{1}{p}\right)^{-1} + O\left(\tau(d)\sqrt{X}\right).$$\end{proof}
We now sketch the proof of Lemma \ref{weird} although it follows very similar lines to the above.

\begin{proof}[Proof of Lemma \ref{weird}]
We again perform a character sum mod 8 to express the congruence condition in the definition of $T''_{k_2, \tilde{k}_2, k_3, \tilde{k}_3}$. After re-arranging this means we can express $T''_{k_2, \tilde{k}_2, k_3, \tilde{k}_3}$ as
$$\frac{1}{4} \sum_{\chi \Mod 8} \chi(\epsilon_1) \sum \limits_{\substack{k_1 \tilde{k}_1 \leq M/k_2k_3\tilde{k}_2\tilde{k}_3 \\ (k_1\tilde{k}_1,2 k_2\tilde{k}_2k_3\tilde{k}_3)=1}}\frac{u(\k)\mu^2(k_1 \tilde{k}_1)}{\tau(k_1 \tilde{k}_1)} \legendre{\tilde{k}_2\tilde{k}_3}{k_1}\legendre{k_2k_3}{\tilde{k}_1} \chi(k_1)\chi(\tilde{k}_1).$$
As noted before, we may think of $u(\k)$ as a character on $k_1$ which can be combined with the other characters to form a new one, called $\widetilde{\chi}_1$. Similarly, the characters on $\tilde{k}_1$ may be combined to form the new character $\widetilde{\chi}_2$. The error term contribution when either of these new characters is non-principal may again be computed using Lemma \ref{FRIW1} and it remains to compute the main term $$\frac{1}{4}\sum \limits_{\substack{k_1\tilde{k}_1 \leq M \\ k_1, \tilde{k_1} \text{ odd}}} \frac{\mu^2(k_1\tilde{k}_1)}{\tau(k_1\tilde{k}_1)}.$$ This is treated exactly as in the earlier proof.\end{proof}

Finally, to estimate \eqref{firstrange} we must compute 
\begin{equation} \label{nsum} \frac{ Mu(\bm{\epsilon})}{4\pi^2}\sum \limits_{\substack{k_1 \leq M \\ k_1 \equiv \epsilon_1 \Mod 8}} \frac{\mu^2(k_1)}{k_1\tau(k_1)}\prod_{p \mid k_1}\left(1+\frac{1}{p}\right)^{-1}.\end{equation}
We again remove the congruence condition with a character sum, we'll first deal with the non-principal characters. We write the sum involving non-principal character $\chi$ as
\begin{align*}
\sum_{k_1 \leq M} \frac{\mu^2(k_1)\chi(k_1)}{k_1 \tau(k_1)} \prod_{p\mid k_1}\left(1 + \frac{1}{p}\right)^{-1} & = \sum_{k_1 \leq M} \frac{\mu^2(k_1)\chi(k_1)}{k_1 \tau(k_1)} \sum_{d \mid k} \mu(d)f(d),
\end{align*} where $f(d) = \prod_{p \mid d} \frac{1}{p+1}.$ Then we may re-arrange the sum to
$$\sum_{d \leq M} \frac{\mu(d)\chi(d)f(d)}{d \tau(d)}\sum \limits_{\substack{k \leq M/d\\ (k,d)=1}} \frac{\mu^2(k)\chi(k)}{k\tau(k)}.$$
For any $C>0$, it follows from Lemma \ref{FRIW1} (via partial summation) that the inner sum above is $\ll \tau(d) (\log M)^{-C}$, while the remaining $d$ sum contributes $\ll \log M$.
 We employ the following general estimate of Friedlander--Iwaniec \cite[Theorem A.5]{FI} to compute the main term arising from the principal character. 

\begin{lemma}\label{OdC}
Suppose $g$ is a multiplicative function supported on squarefree integers such that for some $\kappa> -\frac{1}{2}$ the following hold:
\begin{enumerate}[(i)]
\item $\sum_{p \leq x} g(p) \log p = \kappa \log x + O(1)$,
\item $\prod_{w \leq p <z}\left(1 + \vert g(p) \vert \right) \ll \left( \frac{\log z}{\log w} \right)^{\vert \kappa \vert}$,
\item $\sum_p g(p)^2 \log p < \infty$.
\end{enumerate}
Then we have $$\sum_{n \leq x}g(n) = \frac{(\log x)^{\kappa}}{\Gamma(\kappa+1)} \prod_p \left[ \left( 1- \frac{1}{p}\right)^{\kappa} (1+g(p)) \right] + O\left((\log x)^{\vert \kappa \vert -1}\right).$$
\end{lemma}
That 
$$g(n) =\left\{
	\begin{array}{ll}
		\frac{\mu^2(n)}{\tau(n)n}\prod_{p \mid n} \left(1 + \frac{1}{p}\right)^{-1} & \mbox{if } n \mbox{ odd} \\
		0 & \mbox{otherwise} ,
	\end{array}
\right.   $$ satisfies the assumptions with $\kappa = \frac{1}{2}$ is a simple consequence of Mertens' theorems. Therefore
$$\sum \limits_{\substack{k_1 \leq M \\ k_1 \equiv \epsilon_1 \Mod 8}} \!\!\!\!\!\!g(k_1) = \frac{1}{2\sqrt{2\pi}} \sqrt{\log M} \prod_{p>2} \!\! \left(\!1- \frac{1}{p}\right)^{\!\!\frac{1}{2}} \!\!\left(\!1+ \frac{1}{2(p+1)}\right) + O \left(\!(\log M)^{-\frac{1}{2}}\!\right)\!.$$
Hence we see that \eqref{firstrange} is equal to\begin{equation}\label{bad}\frac{3u(\bm{\epsilon})}{28\pi^2\sqrt{\pi}} M\sqrt{\log M} \prod_p\!\! \left(1-\frac{1}{p}\right)^{\frac{1}{2}}\!\! \left(1+ \frac{1}{2p+2}\right)  +O\left(M \right).\end{equation}

Recall that $M= \frac{\sqrt{X}}{c_{\bm{\delta},\bm{\epsilon},\mu, \alpha, \beta}2^{\mu+\alpha+\beta}}$ so all that remains is to sum over the possible values of $\bm{\delta}, \bm{\epsilon}, \mu, \alpha$ and $\beta$. The following lemma shows that in the range where $\tilde{k}_i \leq V$ for $i=1,2,3$, the $u(\bm{\epsilon})$ factor leads to cancellation in the main term.
\begin{lemma}\label{ucsum}We have
$$\sum_{\delta_2, \delta_3\in \{\pm1\}^2}\sum \limits_{\substack{\mu+ \alpha + \beta \in \{0,1\}\\ \mu,\alpha,\beta \in \{0,1\}}}2^{-(\mu + \alpha + \beta)} \sum_{(\epsilon_1, \delta_2 \epsilon_2, \delta_3 \epsilon_3) \in E(\mu, \alpha, \beta)} \frac{u(\bm{\epsilon} )}{c_{\bm{\delta},\bm{\epsilon},\mu, \alpha, \beta}} = 0.$$
\end{lemma}

\begin{proof}
We start by observing that
$$ \legendre{2^{\mu}}{\epsilon_2\epsilon_3}\legendre{2^{\alpha}}{ \epsilon_1 \epsilon_3}\legendre{2^{\beta}}{ \epsilon_1 \epsilon_2}=1.$$
Indeed this is clearly true when $\mu + \alpha + \beta =0$. If $\mu=1$ then by the definition of $E(1,0,0)$ we must have $\delta_2\epsilon_2 =\delta_3 \epsilon_3$ therefore $$\legendre{2}{ \epsilon_2 \epsilon_3} =1.$$ The other cases follow similarly.\\
Therefore 
\begin{align*}
u(\bm{\epsilon}) &= (-1)^{\nu(\epsilon_1)\nu(\epsilon_2) + \nu( \epsilon_3)\nu(\epsilon_1) + \nu(\epsilon_2) \nu( \epsilon_3)}\legendre{\delta_2}{ \epsilon_1 \epsilon_3}\legendre{\delta_3}{ \epsilon_1 \epsilon_2}\\
&= (-1)^{\nu(\epsilon_1)\nu(\epsilon_2) + \nu( \epsilon_3)\nu(\epsilon_1) + \nu(\epsilon_2) \nu(\epsilon_3)+\nu(\delta_2)\nu(\epsilon_1 \epsilon_3) + \nu(\delta_3)\nu(\epsilon_1 \epsilon_2)}
\end{align*}
Now we just run through all possible values of $\delta_2, \delta_3, \bm{\epsilon}, \mu, \alpha$ and $\beta$ and see what comes out. For simplicity we write $\bm{\epsilon}' := (\epsilon_1, \delta_2\epsilon_2, \delta_3\epsilon_3).$\\

First suppose that $(\delta_2, \delta_3) = (+1,+1)$ then $$u(\bm{\epsilon}) = (-1)^{\nu(\epsilon_1)\nu(\epsilon_2) + \nu( \epsilon_3)\nu(\epsilon_1) + \nu(\epsilon_2) \nu(\epsilon_3)}.$$
By Lemma \ref{HNPcriteria}, we know that at least two components of $\bm{\epsilon}'$ must be equal therefore for some $\mu \in \modstar{8}$ we have $$u(\bm{\epsilon}) = (-1)^{\mu}$$ Hence the sum over these is 0.

Now suppose $(\delta_2, \delta_3) = (+1,-1)$ so that $$u(\bm{\epsilon}) = (-1)^{\nu(\epsilon_1)\nu(\epsilon_2) + \nu( \epsilon_3)\nu(\epsilon_1) + \nu(\epsilon_2) \nu(\epsilon_3)+\nu(\epsilon_1\epsilon_2)}.$$
If $\bm{\epsilon}' \in E_1(\bm{0})$ then $\epsilon_1 \equiv \epsilon_2 \equiv -\epsilon_3  \Mod 4$ so $$ u(\bm{\epsilon}) = (-1)^{\nu(\epsilon_1)}.$$ So again this sums to 0.
Next suppose $\bm{\epsilon}' \in E_2(\bm{0})$ then one of the following cases occurs
\begin{enumerate}[(i)]
\item $\epsilon_1 = \epsilon_2 \equiv -\epsilon_3 \Mod 4$ then $u(\bm{\epsilon}) = (-1)^{\nu(\epsilon_1)+1}$.
\item $\epsilon_2 = \epsilon_3 \equiv - \epsilon_1 \Mod 4$ then $u(\bm{\epsilon}) = (-1)^{\nu(\epsilon_1)+1}$.
\item $\epsilon_3 = \epsilon_1 \equiv - \epsilon_2 \Mod 4$ then $ u(\bm{\epsilon}) = (-1)^{\nu(\epsilon_1)}$.
\end{enumerate}
Each of these cases sums to 0.

If $\bm{\epsilon}' \in E(1,0,0)$ then $\epsilon_2 =- \epsilon_3$ so $$u(\bm{\epsilon}) = (-1)^{\nu(\epsilon_1) + \nu(\epsilon_1 \epsilon_2)}.$$
If $\bm{\epsilon}' \in E(0,1,0)$ then $\epsilon_1 = -\epsilon_3$ so $$u(\bm{\epsilon}') = (-1)^{\nu(\epsilon_2)+\nu(\epsilon_1 \epsilon_2)}.$$
If $\bm{\epsilon}' \in E(0,0,1)$ then $\epsilon_1 = \epsilon_2$ so $$u(\bm{\epsilon}) = (-1)^{\nu(\epsilon_1)}.$$
Again, all of these sum to 0.
The case where $(\delta_2, \delta_3) = (-1, +1)$ is similar.

Finally suppose $(\delta_2, \delta_3) = (-1, -1)$, in which case $$u(\bm{\epsilon}) = (-1)^{\nu(\epsilon_1)\nu(\epsilon_2) + \nu( \epsilon_3)\nu(\epsilon_1) + \nu(\epsilon_2) \nu(\epsilon_3)+\nu(\epsilon_1\epsilon_2)+ \nu(\epsilon_1 \epsilon_3)}.$$
Then for $\bm{\epsilon}' \in E_1(\bm{0})$ we have $$u(\bm{\epsilon}) = (-1)^{\nu(-\epsilon_1)}.$$
For $\bm{\epsilon}' \in E(1,0,0)$ we must have $\epsilon_2 = \epsilon_3$ so $$u(\bm{\epsilon}) = (-1)^{\nu(\epsilon_2)}.$$
For $\bm{\epsilon}' \in E(0,1,0)$ we must have $\epsilon_1 = - \epsilon_3$ so $$u(\bm{\epsilon}) = (-1)^{\nu(\epsilon_2)+\nu(\epsilon_1\epsilon_2)}.$$
For $\bm{\epsilon}' \in E(0,0,1)$ we must have $\epsilon_1 = -\epsilon_2$ so $$u(\bm{\epsilon}) = (-1)^{\nu(\epsilon_3) + \nu(\epsilon_1 \epsilon_3)}.$$
In all of these cases the sum is 0.
\end{proof}

As promised the $u$ term cancels the main term of the sum in this range so the actual main term must be from the range $k_i \leq V$.
Evaluating the sum $T'_{\tilde{k}_1, k_1,k_2,k_3}$ is identical to evaluating $T_{k_1, \tilde{k}_1, \tilde{k}_2,\tilde{k}_3}$ and therefore in analogy to \eqref{bad}, we have that \eqref{secondrange} is equal to
$$ \frac{3}{28\pi^2\sqrt{\pi} } M \sqrt{ \log M} \prod_p\!\! \left(1-\frac{1}{p}\right)^{\frac{1}{2}}\!\! \left(1+ \frac{1}{2p+2}\right)  +O\left(M\right).$$
It just remains to compute that $$\sum_{\delta_2,\delta_3\in\{\pm1\}^2}\sum \limits_{\substack{\mu+ \alpha + \beta \in \{0,1\}\\ \mu,\alpha, \beta \in \{0,1\}}}2^{-(\mu + \alpha + \beta)} \sum_{(\epsilon_1, \delta_2 \epsilon_2, \delta_3 \epsilon_3) \in E(\mu, \alpha, \beta)} \frac{1}{c_{\bm{\delta},\bm{\epsilon},\mu, \alpha, \beta}}=112.$$
Therefore $$\widetilde{S}(X) = \frac{1}{6}\times 112 \times \frac{6}{\pi^2} \times \frac{1}{56\sqrt{2\pi} } \sqrt{X \log X} \prod_p \!\! \left(1-\frac{1}{p}\right)^{\frac{1}{2}}\!\! \left(1+ \frac{1}{2p+2}\right) + O\!\left(\sqrt{X}\right)\!.$$

\end{document}